\documentclass[11pt]{amsart}

\usepackage[margin=1.2in]{geometry} 

\usepackage{amsmath,amsfonts,amsthm,amssymb}
\usepackage{mathtools}
\usepackage{enumitem}
\usepackage{microtype}
\usepackage{hyperref} 

\newtheorem{theorem}{Theorem}[section]
\newtheorem{lemma}[theorem]{Lemma}
\newtheorem{prop}[theorem]{Proposition}
\theoremstyle{definition}
\newtheorem{definition}[theorem]{Definition}
\theoremstyle{remark}
\newtheorem{Remark}[theorem]{Remark}

\newcommand{\Z}{\mathbb{Z}}
\newcommand{\N}{\mathbb{N}}
\newcommand{\R}{\mathbb{R}}
\newcommand{\T}{\mathbb{T}}      % one–dim torus
\newcommand{\C}{\mathbb{C}} 
\newcommand{\ii}{\mathrm i}
\newcommand{\e}{\mathrm e}
\newcommand{\sgn}{\operatorname{sgn}}

\newcommand{\FT}{\mathcal F}
\newcommand{\IFT}{\FT^{-1}}
\newcommand{\wht}[1]{\widehat{#1}}
\newcommand{\Lap}{\Delta_{}}
\newcommand{\Lapd}{\Delta_{s}}

\numberwithin{equation}{section}

\begin{document}

\title{Harmonic extensions on $\mathbb{Z} \times \mathbb{N}$ and a discrete Hilbert transform}

\author{Ljupcho Petrov}
\address{Department of Mathematics, Washington University in St. Louis, One Brookings Drive, St. Louis, MO 63130}
\email{petrov.l@wustl.edu}


\begin{abstract}
For a given boundary sequence $a=(a_n)_{n\in\mathbb{Z}}$, we construct harmonic extensions $U,V:\mathbb{Z}\times\ \mathbb{N}\to \mathbb{R}$ that serve as discrete analogs of the Poisson and conjugate-Poisson integrals. The construction is characterized by: (i) discrete harmonicity with respect to a two-dimensional Laplacian, (ii) a Cauchy-Riemann system, and (iii) boundary values involving a discrete Hilbert transform: $U(n,0)=a_n,\;V(n,0)=(H_{\mathrm d}a)_n$. We compare $H_{\mathrm d}$ to the Riesz-Titchmarsh transform and prove weak-type $(1,1)$ and $\ell^{p}$ bounds for $p>1$. We also extend the constructions to harmonic extensions on $\mathbb{Z}^s \times \mathbb{N}$. These results provide a discrete harmonic-analytic model analogous to the classical theory.
\end{abstract}

\maketitle

\section{Introduction}

\subsection*{The continuous model on $\R\times\R_{+}$}

On the real line, the Hilbert transform $Hf$ arises as the boundary value of the
conjugate harmonic function $Q[f]\;=\;\Im F[f]$, where
\[
F[f](z)=\frac{\ii}{\pi}\int_{\R}\frac{f(t)}{z-t}\,dt,\quad z=x+\ii y\in\C_{+}.
\]
Precisely, with the Poisson extension $\Re F[f]=P[f]$, for a.e. $x\in\R$ we have \cite{stein-book}
\[
  \lim_{y\to0^{+}}P[f](x+\ii y)=f(x),\qquad
  \lim_{y\to0^{+}}Q[f](x+\ii y)=Hf(x).
\]

\subsection*{Discrete–continuous settings on $\Z\times\R_{+}$}

In \cite{CiaurriEtAl2017}, Ciaurri, Gillespie, Roncal, Torrea, and Varona show that the discrete Riesz-Titchmarsh transforms arise as a limit of conjugate harmonic extensions for the discrete fractional Laplacian on $\Z\times\R_{+}$. They start from the one–dimensional discrete Laplacian
\[
  \Delta_{x} f(n)=f(n{+}1)-2f(n)+f(n{-}1),\qquad n\in\Z,
\]
and construct the heat semigroup $\{W_y\}_{y>0}$ with explicit kernel
\[
  W_y f(n)=\sum_{m\in\Z} e^{-2y}\,I_{\,n-m}(2y)\,f(m),
\]
where $I_k$ is the modified Bessel function. The function $u(n,y):=W_y f(n)$ solves the semi-discrete heat equation $\partial_y u(n,y)=\Delta_{x} u(n,y)$ with the boundary sequence on $\Z$ being $u(\cdot,0)=f$. Here, $\partial_y$ is the continuous derivative in $y$. From the heat semigroup, they obtain the Poisson semigroup by subordination and use it to define the fractional powers $(-\Delta_{x})^\sigma$ for $0<\sigma<1$. They utilize the difference factorization of the Laplacian: $\Delta_{x}= \delta_h^{+} \, \delta_h^{-}$, with 
$$  \delta_h^{+} f(n)=f(n{+}1)-f(n) \quad \text{and} \quad  \delta_h^{-}f(n)=f(n)-f(n{-}1).$$
Since $(-\Delta_{x})^{-1/2}$ is not well defined on $\Z$, they set
\[
  R^+=\lim_{\alpha\to 1/2^+}  \delta_h^{+} \,(-\Delta_{x})^{-\alpha},\qquad
  R^{-}=\lim_{\alpha\to 1/2^+}  \delta_h^{-}\,(-\Delta_{x})^{-\alpha},
\]
and identify these as convolution operators with the kernels of the Riesz–Titchmarsh transforms $H^{\pm}$ \cite{Titchmarsh1926Reciprocal,Titchmarsh1928Inequality,Riesz1928Conjuguees}: 

$$K^+(n)=\frac{1}{\pi}\frac{1}{n+\tfrac12}, \qquad K^-(n)=\frac{1}{\pi}\frac{1}{n-\tfrac12}.$$  
The associated conjugate harmonic functions $Q_y f:=R^+ P_y f$ and $Q_y^{*} f:=R^{-}P_y f$ satisfy Cauchy–Riemann type identities involving $ \delta_h^{+}$, $ \delta_h^{-}$, and the continuous derivative $\partial_y$. The boundary limits recover $R^+$ and $R^{-}$. 

In \cite{ArcozziDomelevoPetermichl2020}, Arcozzi, Domelevo, and Petermichl develop a probabilistic and semigroup-based account of the centered Hilbert transform
\[
  H_c=\tfrac12\,(H^{+}+H^{-}),
\]
built from the transforms $H^{\pm}$ determined by the discrete derivatives $\delta_h^{+}$ and $\delta_h^{-}$ via the factorization
\[
  H^{+}\circ\sqrt{-\Delta_x}= \delta_h^{+},\qquad H^{-}\circ\sqrt{-\Delta_x}= \delta_h^{-}. 
\]
On the upper half-space $\Z\times\R_{+}$, they take $A=\sqrt{-\Delta_x}$ and define the semi-discrete Poisson semigroup $P_y=e^{-yA}$, $y>0$.  The Poisson extension $u_f(y,x):=(P_y f)(x)$ solves the semi-discrete Laplace equation
\[
  \Delta_x u_f+\partial_y^2 u_f=0\quad\text{on }\Z\times\R_{+}.
\]
They also present Cauchy–Riemann identities relating $u_f$ to the Poisson extension of $H_cf$.  

\subsection*{Discrete analytic functions}

A classical approach in discrete complex analysis is to define a notion of “analyticity” on lattices via discrete Cauchy–Riemann relations.  On the integer lattice, one can define a pair of real-valued functions whose finite differences satisfy relations mimicking the continuous Cauchy-Riemann equations.  This was pursued early by Ferrand \cite{Ferrand1944} and Isaacs \cite{Isaacs1941}, and later by Duffin \cite{Duffin1953, Duffin1956}. 

\subsection*{Main contributions of this paper}

In contrast to the semi–discrete settings described above, we impose the fully discrete
five–point Laplace equation on $\Z\times\N$. For a given boundary sequence $a=(a_n)_{n\in\Z}$, we build a pair of functions $U,V:\Z\times\N\longrightarrow\R$,
that play the role of the Poisson and conjugate–Poisson extensions on the fully discrete upper half–lattice $\Z\times\N$. The conjugate extension is forced by requiring a lattice Cauchy–Riemann system to hold and gives rise to a discrete Hilbert transform $H_{\mathrm{d}}$. We are not aware of any previous work that realizes a discrete Hilbert transform by means of extensions to $\Z \times \N$ of this type.  

\begin{definition}\label{def: My_H_d}
The discrete Hilbert transform $H_{\mathrm{d}}$ of a sequence $a \in\ell^p(\Z)$ with $p \ge 1$ is defined as
  \[
    (H_{\mathrm{d}}a)_n
      =\IFT\!\bigl[ (-\ii\,\sgn\theta)\, \rho(\theta)^{1/2}\e^{\ii\theta/2} \wht a(\theta)\bigr](n).
  \]    
\end{definition}

With this definition, our main result is the following.

\begin{theorem}[Main theorem]\label{thm:main}
Let $p > 1$ and $a\in\ell^p(\Z)$. Define
  \[
      U(n,k)=\sum_{m\in\Z} P_k(n-m)\,a_m,
      \qquad
      V\!\left(n,k\right)=\sum_{m\in\Z} Q_k(n-m)\,a_m .
  \]
Then:
  \begin{enumerate}[label=\textup{(\alph*)}]
    \item (Discrete harmonicity). Both $U$ and $V$ are harmonic with respect to the two–dimensional five–point Laplacian
\[
(\Lap{F})(n,k)=F(n{+}1,k)+F(n{-}1,k)+F(n,k{+}1)+F(n,k{-}1)-4F(n,k), \quad k\ge1. \]
    \item (Boundary values at height $k=0$). At height $k=0$, the extensions satisfy
\[
U(n,0)=a_n,\qquad
V(n,0)=(H_{\mathrm{d}}a)_n,
\]
where $H_{\mathrm{d}}$ is the discrete Hilbert transform defined above.
    \item (Cauchy–Riemann relations). $U$ and $V$ satisfy the discrete Cauchy–Riemann relations
      of Proposition~\ref{prop:CR-new}.
  \end{enumerate}
\end{theorem}

This construction yields the operator $H_{\mathrm{d}}$ with Fourier multiplier
\[
  m_{\mathrm d}(\theta)=(-\ii\,\sgn\theta)\,\rho(\theta)^{1/2}e^{\ii\theta/2},
\]
where $\rho+\rho^{-1}=4-2\cos\theta$.  

As proven below, the extensions $U$ and $V$ are real-valued when the sequence is real-valued.

\subsection*{Preliminaries}
$\,$

\begin{enumerate}[label=\textup{(\roman*)}]
\item For a sequence $ a \in \ell^1(\mathbb{Z}) $, we define its horizontal Fourier transform on the torus $ \mathbb{T} = (-\pi, \pi] $ by 
\[
  \mathcal{F}(a)(\theta)=\widehat{a}(\theta)
  = \sum_{n \in \mathbb{Z}} a_n \, e^{-\ii n\theta}, 
  \qquad \theta \in \mathbb{T}.
\]
For $ g \in L^1(\mathbb{T}) $, the inverse Fourier transform is given by 
\[
  \mathcal{F}^{-1}[g](n)
  = \frac{1}{2\pi} \int_{-\pi}^{\pi} g(\theta) \, e^{\ii n\theta} \, d\theta,
  \qquad n \in \mathbb{Z}.
\]

\item For a function $F:\Z\times\N\to\mathbb C$, we define the discrete two-dimensional Laplacian as follows:
\[
  (\Lap F)(n,k)=F(n+1,k)+F(n-1,k)+F(n,k+1)+F(n,k-1)-4F(n,k),
  \quad k\ge 1.
\] 

\item We call $F$ (discrete) harmonic if $\Lap F\equiv0$.  
Boundary values are evaluated at the height $k=0$. 
\end{enumerate} 

\section{The Poisson and Conjugate-Poisson Extensions}

\subsection*{Solutions of the Laplace equation}

Seeking solutions of $\Lap U=0$, we fix $k$ and take the horizontal Fourier transform in $n$:
\[
  \widehat U(\theta,k)=\sum_{n\in\mathbb Z}U(n,k)\,\mathrm e^{-\mathrm i n\theta}.
\]
We repeatedly use the shift rule:
\begin{equation}\label{eq:shift}
  \mathcal F\!\big[U(n\pm1,k)\big]
   =\mathrm e^{\pm\mathrm i\theta}\,\widehat U(\theta,k).
\end{equation}
Using \eqref{eq:shift} and linearity,
\begin{align*}
  0=\mathcal F[\Lap{U}](\theta,k)
  &=\big(\mathrm e^{\mathrm i\theta}+\mathrm e^{-\mathrm i\theta}-4\big)\,\widehat U(\theta,k)
     +\widehat U(\theta,k+1)+\widehat U(\theta,k-1)\\
  &=\big(2\cos\theta-4\big)\,\widehat U(\theta,k)
     +\widehat U(\theta,k+1)+\widehat U(\theta,k-1).
\end{align*}
Thus, for each fixed $\theta\in(-\pi,\pi]$, the functions $k\mapsto \widehat U(\theta,k)$ satisfy the homogeneous linear recurrence
\begin{equation}\label{eq:recurrence-k}
  \widehat U(\theta,k+1)+\widehat U(\theta,k-1)+(2\cos\theta-4)\,\widehat U(\theta,k)=0,\qquad k\ge1.
\end{equation}

We look for separated solutions of the form
\[
  \widehat U(\theta,k)=C(\theta)\,\rho(\theta)^{\,k},
\]
with $C(\theta)$ independent of $k$. Substituting into \eqref{eq:recurrence-k} gives
\begin{equation}\label{eq:recurrence_for_rho}
\begin{split}
  \rho+\rho^{-1}+2\cos\theta-4=0.
\end{split}  
\end{equation}
Solving this quadratic equation yields two real reciprocal roots. We select the root in $(0,1]$:
\begin{equation}\label{eq:rho-again}
  \rho(\theta)=2-\cos\theta-\sqrt{(2-\cos\theta)^2-1}.
\end{equation}
This choice yields normalized, positive kernels $P_{k}\in\ell^{1}(\Z)$ with $\sum_{n}P_{k}(n)=1$ and produces the unique $\ell^{p}$-bounded solution of the Dirichlet problem on $\Z\times\N$ (Lemma \ref{lem:positivity}).

\subsection*{The Poisson extension}

Let the boundary sequence be $U(n,0)=a_n$, so the Fourier transform at $k=0$ is
$  \widehat U(\theta,0)=\widehat a(\theta). $ Evaluating $ \widehat U(\theta,k)=C(\theta)\,\rho(\theta)^{\,k}$ at $k=0$ fixes $C(\theta)=\widehat a(\theta)$; hence,
\begin{equation*}
  \widehat U(\theta,k)=\rho(\theta)^{\,k}\,\widehat a(\theta),\qquad k\ge0.
\end{equation*}
By taking the inverse transform in $n$,
\begin{equation}\label{eq:U-inversion}
  U(n,k)=\frac1{2\pi}\int_{-\pi}^{\pi}\rho(\theta)^{\,k}\,\widehat a(\theta)\,\mathrm e^{\mathrm i n\theta}\,d\theta.
\end{equation}

\begin{definition}[Discrete Poisson kernel]
  For $k\in\N$ and $n\in\Z$, let
  \begin{equation}\label{eq:Pk}
    P_k(n)=\frac1{2\pi}\int_{-\pi}^{\pi} \rho(\theta)^{\,k}\,
            \e^{\ii n\theta}\,d\theta .
  \end{equation}
\end{definition}

For $a \in \ell^p(\Z)$ with $p \ge 1$, we have
\[
  U(\cdot,k)=P_k*a,\qquad
  U(n,k)=\sum_{m\in\mathbb Z}P_k(n-m)\,a_m.
\]
Indeed, by Lemma \ref{lem:positivity} below, the convolution is well-defined and in $\ell^p(\Z)$ for every $a \in \ell^p(\Z)$ with $p \ge 1$.

\begin{lemma}[Elementary properties of $P_k$]\label{lem:positivity}
The discrete Poisson kernel is positive, symmetric, and normalized:
\begin{itemize}
\item[(i)] $P_k(n)\in[0,1]$ for all $n\in\Z$, $k\ge1$.

\item[(ii)] $P_k(-n)=P_k(n)$.

\item[(iii)] $\displaystyle \sum_{n\in\Z} P_k(n)=1$ for every
$k\ge1$. In particular, $P_k \in \ell^1(\Z)$.
\end{itemize}

\end{lemma}

\begin{proof} $\,$

\begin{itemize}
\item[(i)] Using $\rho(-\theta)=\rho(\theta)$, we get that
\begin{equation}\label{eq: Pk_real}
    P_k(n)=\frac1{\pi}\int_{0}^{\pi} \rho(\theta)^{\,k}\,
    \cos{(n\theta)}\,d\theta.
  \end{equation}
Therefore, $P_k(n)$ is real. Let $a_n\ge 0$ and let $U$ solve $\Lap{U}=0$ on $\Z\times \N$ with boundary $U(\cdot,0)=a$. By the maximum principle, $U\ge 0$. Taking $a=\delta_0$ (the unit Dirac delta mass), gives $P_k(n)=U(n,k)\ge 0$ for all $n$. Now, the fact that $P_k(n)\in[0,1]$ follows from part (iii).

\item[(ii)] This follows immediately from \eqref{eq: Pk_real}.

\item[(iii)] Let the Fej\'er kernel be: $F_N(\theta)=\sum_{|n| \le N}(1-|n|/(N+1))e^{\ii n\theta}$. Because $\rho$ is continuous with $\rho(0)=1$, by Fej\'er's theorem \cite{Korner}, 
\[
  \sum_{|n| \le N}\Bigl(1-\frac{|n|}{N+1}\Bigr)P_k(n)
  =\frac{1}{2\pi}\int_{-\pi}^{\pi}\rho(\theta)^{k}F_N(\theta)\,d\theta
  \xrightarrow[N\to\infty]{}\rho(0)^{k}=1.
\]
If we denote \[
  \sigma_N=\frac{1}{N+1}\sum_{r=0}^{N}S_r,
\]
with $S_r=\sum_{|n|\le r}P_k(n)
$, then we just proved that $\displaystyle \lim_{N \to \infty} \sigma_N=1$. 

By the non-negativity of $P_k(n)$, we have $S_r\le S_N$ for $r\le N$, so
$\sigma_N\le S_N$;
hence, $\liminf_{N}S_N\ge\lim_{N}\sigma_N=1$. Next, fix $M$ and take $N\ge M$.  Then,
\[
  \sigma_N=\frac{1}{N+1}\Big(\sum_{r=0}^{M-1}S_r+\sum_{r=M}^{N}S_r\Big)
           \ \ge\ \frac{N-M+1}{N+1}\,S_M.
\]
Letting $N\to\infty$ gives $1\ge S_M$.  Since $M$ is arbitrary, $\limsup_{N}S_N\le1$.

We conclude that $\displaystyle \lim_{N \to \infty}S_N=\displaystyle \sum_{n\in\Z} P_k(n)=1$.
\end{itemize}
\end{proof}

\subsection*{The conjugate–Poisson kernel}

\begin{definition}
On the Fourier side, we define the conjugate-Poisson kernel as follows:
\begin{equation}\label{eq:V-d}
   \widehat{V}(\theta,k)
   =\bigl(-\ii\,\sgn\theta\bigr)\,
     \rho(\theta)^{\,k+\frac12}\,
     \e^{\ii\theta/2}\,
     \widehat a(\theta),
   \qquad k\ge 0.
\end{equation}
\end{definition}

The factor $\e^{\ii\theta/2}$ shifts $V$ half a unit horizontally, locating $V$
at $(n+\tfrac12,k)$. Let
$
   V(n+\tfrac12,k)=\IFT[\widehat V(\,\cdot\,,k)](n)
$, and let us define 
$$
Q_k(n):=\frac{1}{2\pi}\int_{-\pi}^{\pi}
          \bigl(-\ii\,\sgn\theta\bigr)\,\rho(\theta)^{k+\frac12}\,e^{\ii(n+\frac12)\theta}\,d\theta.
$$
We note that $Q_k(n)$ is real-valued, which follows from $\rho(-\theta)=\rho(\theta)$. As a convolution formula, for $a \in \ell^p(\Z)$ with $p>1$, we have 
$$
V\!\left(n+\tfrac12,k\right)=\sum_{m\in\Z}Q_k(n-m)\,a_m.
$$ 

For notational convenience, we may relabel this as $V(n,k)$, so that both $U$ and $V$ are indexed on $\Z\times\N$. As we will see, this definition is imposed by the Laplacian and our requirement that $U$ and $V$ satisfy a Cauchy-Riemann system. 

\begin{lemma}[Harmonicity]\label{lem: harmonicity}
  For each fixed sequence $a\in\ell^p(\Z)$ with $p>1$, the extensions
  \[
    U(n,k)=(P_k\ast a)_n,\qquad 
    V(n,k)=(Q_k\ast a)_n
  \]
  satisfy $\Lap U=\Lap V=0$ on $\Z\times\N$.
\end{lemma}

\begin{proof}
$U$ is harmonic by construction. Taking horizontal Fourier transforms of $\Lap V=0$ reduces to the same recurrence relation \eqref{eq:recurrence_for_rho}. 
\end{proof}

\section{Discrete Cauchy–Riemann system}

To prove that $U$ and $V$ form a Cauchy-Riemann system, we will need the following algebraic identity.

\begin{lemma}\label{lem:rho-id}
For every $\theta\in(-\pi,\pi]$ the quantity
$\rho(\theta)$ from~\eqref{eq:recurrence_for_rho} satisfies
\begin{equation}\label{eq:rho-identity-correct}
     \rho(\theta)-1
       \;=\;
       -\,\bigl(\e^{\ii\theta}-1\bigr)\,
         \bigl(-\ii\,\sgn\theta\bigr)\,
         \e^{-\ii\theta/2}\,
         \rho(\theta)^{1/2}.
\end{equation}
\end{lemma}

\begin{proof}
We rewrite $\rho+\rho^{-1}=4-2\cos\theta$ as follows:
\begin{align*}
  (\rho-1)+\bigl(\rho^{-1}-1\bigr)=&2(1-\cos\theta)
           \\ =&\,4\sin^2(\tfrac{\theta}{2}).
\end{align*}
Thus,
$
\displaystyle \frac{(\rho-1)^{2}}{\rho}
        = 4\sin^{2}\!\bigl(\tfrac{\theta}{2}\bigr),
$
which implies that
\begin{equation}\label{eq:square}
   \bigl(\rho-1\bigr)^{2}
        \;=\;
        4\sin^{2}\!\bigl(\tfrac{\theta}{2}\bigr)\,\rho.
\end{equation}
Since $0<\rho<1$, we have $\rho-1<0$.
Moreover, $\sin(\theta/2)$ has the same sign as~$\theta$, so ~\eqref{eq:square} becomes
\begin{equation*}
   \rho-1
     = -\,2\bigl|\sin\tfrac{\theta}{2}\bigr|\,
         \rho^{1/2}
     = -\,2\sin\tfrac{\theta}{2}\;
         \bigl(\sgn\theta\bigr)\,
         \rho^{1/2}.
\end{equation*}
Using
\[
  \e^{\ii\theta}-1
      = 2\ii\sin\tfrac{\theta}{2}\;\e^{\ii\theta/2},
\]
we get

\[
    \rho-1 =
         -\,( \e^{\ii\theta}-1)\,(-\ii\,\sgn\theta)\,\rho^{1/2}\e^{-\ii\theta/2}.
\]
\end{proof}

Define the horizontal and vertical difference operators as:
\begin{align*}
  \delta_h^{+}F(n,k)&=F(n+1,k)-F(n,k),\\[2pt]
  \delta_h^{-}F(n,k)&=F(n,k)-F(n-1,k),\\[2pt]
  \partial_k^{+}F(n,k)&=F(n,k+1)-F(n,k),\\[2pt]
  \partial_k^{-}F(n,k)&=F(n,k)-F(n,k-1).
\end{align*}

Using the properties of the Fourier transform, we obtain the following proposition. 

\begin{prop}[Discrete Cauchy-Riemann system]\label{prop:CR-new}
For every $n\in\Z$ and $k\ge1$,
  \begin{align}
    \;\;\delta_h^+ U(n,k)&=\partial_k^{-}V(n,k),\label{eq:CR-hor} \\[4pt]
      \partial_k^{+}U(n,k) 
      &=-\,\delta_h^- V(n,k).\label{eq:CR-vert} 
\end{align}
\end{prop}

\begin{proof}
After taking the Fourier transform $\FT$, \eqref{eq:CR-hor} becomes
\[
  (\e^{\ii\theta}-1)\,\rho^k\widehat a
  \;=\;\big(1-\rho^{-1}\big)\,(-\ii\,\sgn\theta)\,\rho^{k+\frac12}\e^{\ii\theta/2}\widehat a,
\]
which reduces to the identity from Lemma \ref{lem:rho-id}. The same happens when we take the Fourier transform of \eqref{eq:CR-vert}:
\[
  (\rho-1)\,\rho^k\widehat a
  \;=\;-\big(1-\e^{-\ii\theta}\big)\,(-\ii\,\sgn\theta)\,\rho^{k+\frac12}\e^{\ii\theta/2}\,\widehat a.
\]
\end{proof}

\section{A Discrete Hilbert Transform and Properties}

\subsection*{A discrete Hilbert transform and the main theorem}

The above constructions lead us to the following definition of a discrete Hilbert transform, $H_\mathrm{d}$, for a sequence $(a_n)_{n \in \Z}$ (Definition \ref{def: My_H_d}):
  \[
    (H_{\mathrm{d}}a)_n
      =\IFT\!\bigl[ (-\ii\,\sgn\theta)\, \rho(\theta)^{1/2}\e^{\ii\theta/2} \wht a(\theta)\bigr](n).
  \]    

We have thus proven all assertions of Theorem \ref{thm:main}, and in this section, we will consider several properties of $H_{\mathrm d}$.

The discrete Hilbert transform we are considering has the Fourier multiplier:
\[ m_{d}(\theta)=-\ii\sgn\theta\,\rho(\theta)^{1/2}\,\e^{\ii\theta/2}.
\]
This multiplier was forced by the Cauchy-Riemann system we considered. Nevertheless, since the Fourier multiplier is bounded, we immediately get boundedness of the discrete Hilbert transform on $\ell^2(\mathbb Z)$. We prove stronger boundedness properties below.

\subsection*{Properties of the discrete Hilbert transform}

The kernel of our discrete Hilbert transform model is given by
\begin{equation}\label{eq:K_d(n)}
\begin{split}
K_{\mathrm{d}}(n)
  = &\frac{1}{2\pi}\!\int_{-\pi}^{\pi}\!(-\ii\,\sgn\theta)\,\rho(\theta)^{1/2}\,\e^{\ii(n+\frac12)\theta}\,d\theta
  \\ \;=&\; \frac{1}{\pi}\!\int_{0}^{\pi}\!\rho(\theta)^{1/2}\,\sin\!\bigl((n+\tfrac12)\theta\bigr)\,d\theta,
\end{split}
\end{equation}
so $H_{\mathrm{d}}a=K_{\mathrm{d}}*a$.

By using the Taylor expansion of $\rho(\theta)$, it is possible to integrate and express $K_{\mathrm {d}}$ explicitly as an infinite series in $n$. However, the above integral representation of $K_{\mathrm{d}}$ is more useful when exploring properties and weighted bounds of $H_{\mathrm{d}}$, and we will not give the more complicated formula for it in terms of a series representation. 

\begin{prop}[Kernel antisymmetry and cancellation]\label{prop:kernel-symmetry}
The kernel $K_{\mathrm{d}}$ is real-valued. Moreover, it is antisymmetric:
\[
  K_{\mathrm{d}}(-n-1)=-K_{\mathrm{d}}(n), \quad n\in\Z,
\]
and in particular, $\sum_{n\in\Z}K_{\mathrm{d}}(n)=0$ (understood as pairwise cancellation between pairs $n$ and $-n-1$). 
\end{prop}

\begin{proof}
Formula \eqref{eq:K_d(n)} shows that the kernel is real-valued. Let $m_d(\theta)=(-\ii\,\sgn\theta)\rho^{1/2}\e^{\ii\theta/2}$ be the Fourier multiplier. Then,
\begin{align*}
  m_d(-\theta)=& (-\ii\,\sgn(-\theta))\rho^{1/2}\e^{-\ii\theta/2}
            \\= & \ii\,\sgn\theta\,\rho^{1/2}\e^{-\ii\theta/2}
            \\ = &-\,\e^{-\ii\theta}\,m_d(\theta).
\end{align*}
Thus, we get that
\begin{align*}
  K_{\mathrm{d}}(-n-1)
  =&\frac{1}{2\pi}\!\int_{-\pi}^{\pi} m_d(\theta)\,\e^{-\ii(n+1)\theta}\,d\theta \\
  =&-\frac{1}{2\pi}\!\int_{-\pi}^{\pi} m_d(-\theta)\, e^{\ii \theta} \e^{-\ii n\theta}\,\e^{-\ii\theta}\,d\theta
  \\
  =&-\frac{1}{2\pi}\!\int_{-\pi}^{\pi} m_d(\theta)\,\e^{\ii n\theta}\,d\theta
  \\ = &-K_{\mathrm{d}}(n).
\end{align*}
The pairwise cancellation of the series $\sum_{n}K_{\mathrm{d}}(n)$ follows.
\end{proof}

\subsection*{Asymptotics of $K_{\mathrm{d}}(n)$ and $\ell^p(w)$ bounds}

We next compare $H_{\mathrm d}$ to the kernel of the Riesz–Titchmarsh transform 
\begin{align*}
  K_{H^+}(n)=& \frac{1}{2\pi}\!\int_{-\pi}^{\pi}\!(-\ii\,\sgn\theta)\,\,\e^{\ii(n+\frac12)\theta}\,d\theta \\ =&\;\frac{1}{\pi}\!\int_{0}^{\pi}\!\sin\!\bigl((n+\tfrac12)\theta\bigr)\,d\theta
  \\   \;=&\frac{1}{\pi}\frac{1}{n+\tfrac12}.
\end{align*}

We define the difference kernel $D(n):=K_{\mathrm{d}}(n)-K_{H^{+}}(n)$, and prove that
$D\in\ell^1(\Z)$, while giving an explicit bound.

Weighted $\ell^p(w)$ bounds are an immediate corollary. A weight is a function $w:\Z\to(0,\infty)$.  For integers $M\le N$, write
$I=[M,N]\cap\Z$ and $|I|=N-M+1$.  For $1<p<\infty$, the Muckenhoupt class $A_{p}$ \cite{HuntMuckenhouptWheeden1973} consists of all
$w$ for which there exists $C<\infty$ such that, for every pair $M\le N$,
\begin{equation}
  \Big(\sum_{k=M}^{N} w(k)\Big)\,
  \Big(\sum_{k=M}^{N} w(k)^{-1/(p-1)}\Big)^{p-1}
  \;\le\; C\,|I|^{p}.
\end{equation}
For $A_{1}$, one requires that, for all $M\le N$,
\begin{equation}
  \Big(\sum_{k=M}^{N} w(k)\Big)\,
  \max_{k\in I} \frac{1}{w(k)}
  \;\le\; C\,|I|.
\end{equation}

Now, if we set
\[
f(\theta):=\rho(\theta)^{1/2}-1,\qquad \theta\in[0,\pi],
\]
the difference kernel is given by
\begin{equation}\label{eq:Dn-def}
  D(n)=K_{\mathrm{d}}(n)-K_{H^+}(n)
  \;=\;\frac{1}{\pi}\int_{0}^{\pi} f(\theta)\,\sin\!\bigl((n+\tfrac12)\theta\bigr)\,d\theta.
\end{equation}
We prove that $D(n)=O(n^{-2})$; hence,
$K_{\mathrm{d}}(n)=\frac{1}{\pi(n+\frac12)}+O(n^{-2})$ as $|n|\to\infty$. 

Starting from Equation \eqref{eq:square}: $ (\rho-1)^{2}=4\sin^{2}(\tfrac{\theta}{2})\,\rho,
$
for $\theta\in[0,\pi]$, we get that
\begin{equation*}
  \rho(\theta)^{1/2}
   = \sqrt{1+\sin^{2}\!\tfrac{\theta}{2}}\;-\;\sin\tfrac{\theta}{2}.
\end{equation*}
Expanding around $\theta=0$, we get that as $\theta\to 0$:
\begin{align*}
  \rho(\theta)^{1/2}
  &= \Bigl(1+\tfrac12\sin^{2}\tfrac{\theta}{2}+O(\theta^{4})\Bigr)
     -\sin\tfrac{\theta}{2}
   \notag\\
  &= 1 - \frac{\theta}{2} + \frac{\theta^{2}}{8} + O(\theta^{3}),
\end{align*}
so that
\begin{equation}\label{eq:f-expansion}
  f(\theta)=\rho(\theta)^{1/2}-1
  = -\frac{\theta}{2} + \frac{\theta^{2}}{8} + O(\theta^{3}).
\end{equation}
In particular, $f(0)=0$, $f'(0)=-\tfrac12$, and $f''(0)=\frac14$. Using the explicit formula for $f(\theta)$, one also finds that $f'(\pi)=0$ and that $f''\in C[0,\pi]$.

\begin{lemma}\label{prop:strong-closeness}
The difference kernel
$D=K_{\mathrm{d}}-K_{H^+}$ belongs to $\ell^1(\Z)$, the asymptotic formula

$$K_{\mathrm{d}}(n)=\dfrac{1}{\pi(n+\frac12)}+O(n^{-2})$$ 
is valid as $|n|\to\infty$, and for every $1\le p\le\infty$,
\[
  \|H_{\mathrm{d}}-H^+\|_{\ell^p\to\ell^p}\;\le\;\|D\|_{\ell^1}\;<\infty.
\]
\end{lemma}

\begin{proof}
Integrating by parts twice in \eqref{eq:Dn-def}, we obtain 
\begin{equation}\label{eq:Dn-bound}
\begin{split}
  |D(n)|
  \;=& \;\frac{1}{\pi}\Bigl|\int_{0}^{\pi} f(\theta)\,\sin\!\bigl((n+\tfrac12)\theta\bigr)\,d\theta\Bigr|
  \\   \; \le & \; \frac{1}{\pi\left(n+\frac12\right)^{2}}\int_{0}^{\pi} |f''(\theta)|\,d\theta\\[2pt]
  \; =  &\; \frac{C_0}{\left(n+\frac12\right)^{2}} \sim \frac{1}{n^2}.
\end{split}  
\end{equation}
Here, we used the fact that $f \in C^2[0,\pi]$, so that $\int_{0}^{\pi} |f''(\theta)|\,d\theta \sim 1$.
This proves that $D(n)=O(n^{-2})$, which implies that
$K_{\mathrm{d}}(n)=\frac{1}{\pi(n+\frac12)}+O(n^{-2})$ as $|n|\to\infty$. Moreover, by using $\displaystyle \sum_{n \in \Z} \frac{1}{\left( n+\frac 12\right)^2}=\pi^2$, we find that $\|D\|_{\ell^1} < \infty$.

\noindent The difference
\[
  (H_{\mathrm{d}}-H^+)a = D*a
\]
is convolution with the $\ell^1$ kernel $D$, and the final claimed bound follows from the discrete Young’s inequality.
\end{proof}

Using the fact that $H^+$ is a Calder\'on-Zygmund operator, and hence bounded on $\ell^p(\Z)$ and $\ell^p(\Z; w)$ when $w \in A_p$, we obtain the following results.

\begin{prop}[Weak-type $(1,1)$ and $\ell^p$ bounds]\label{prop:weak11}
For $p>1$, we have
\[
\|H_{\mathrm{d}}a\|_{\ell^p}\;\le\;\|H^{+}a\|_{\ell^p}+\|D\|_{\ell^1}\,\|a\|_{\ell^p}.
\]
In addition, since $H^{+}$ is of weak type $(1,1)$,
\[
  \#\big\{n:\ |H_{\mathrm{d}}a(n)|>\lambda\big\}
  \;\le\;\frac{C\,\|a\|_{\ell^1}}{\lambda}\;+\;\frac{\|D\|_{\ell^1}\,\|a\|_{\ell^1}}{\lambda}
  \qquad(\lambda>0),
\]
so $H_{\mathrm{d}}$ is weak type $(1,1)$ and bounded on $\ell^p$ for $1<p<\infty$.
\end{prop}

\begin{proof}
Write $H_{\mathrm{d}}=H^{+}+D*$. Young’s inequality gives
$\|D*a\|_{\ell^p}\le\|D\|_{\ell^1}\|a\|_{\ell^p}$, and
$\|H^{+}\|_{\ell^p\to\ell^p}<\infty$ for $1<p<\infty$. For weak $(1,1)$,
Chebyshev and $\ell^1$–boundedness of $D*$ yield
$\#\{|D*a|>\lambda\}\le\|D*a\|_{\ell^1}/\lambda\le \|D\|_{\ell^1}\|a\|_{\ell^1}/\lambda$,
and we combine with the weak $(1,1)$ bound for $H^{+}$.
\end{proof}

\begin{Remark}
By the proof of Lemma \ref{prop:strong-closeness}, the function $K(n):=C_0 \,(|n|+\frac12)^{-2} \in \ell^1(\Z)$ is a radially decreasing majorant of $D$. Then, $(|D|*|a|)(n)\ \le\ C\,M a(n)$,
where $M$ is the discrete Hardy–Littlewood maximal operator. The weighted bounds for $M$ give \cite{stein-book}
\begin{align*}
  \|D*a\|_{\ell^{p}(w)}
   & \lesssim\ [w]_{A_{p}}^{1/(p-1)}\ \|a\|_{\ell^{p}(w)},
    \qquad 1<p<\infty,\\[2pt]
  w\big(\{|D*a|>\lambda\}\big)
   & \lesssim\ \frac{[w]_{A_{1}}}{\lambda}\,\|a\|_{\ell^{1}(w)},
    \qquad \lambda>0.
\end{align*}
Thus, $H_{\mathrm{d}}$ is bounded on $\ell^p(w)$, if $w \in A_p$, strongly for $1<p<\infty$ and weakly for $p=1$. 
\end{Remark}

\begin{Remark}
We will compare $H_{\mathrm{d}}$ and $H^{+}$ at the level of their Fourier multipliers:
\[
  m_{\mathrm{d}}(\theta)
   =(-\ii\,\sgn\theta)\,\rho(\theta)^{1/2}\,\e^{\ii\theta/2},\qquad
  m_{H^+}(\theta)=(-\ii\,\sgn\theta)\,\e^{\ii\theta/2}.
\]
The expansion \eqref{eq:f-expansion} gives
\[
  m_{\mathrm{d}}(\theta)-m_{+}(\theta)
  =(-\ii\,\sgn\theta)\bigl(\rho^{1/2}(\theta)-1\bigr)\e^{\ii\theta/2}
  = \frac{\ii}{2}\,|\theta| + O(\theta^{2})\qquad(\theta\to0).
\]
Thus, the difference multiplier vanishes to first order at the origin. On the circle,
a $C^{1}$ multiplier with a simple zero at $0$ has Fourier coefficients decaying as
$O(|n|^{-2})$. This is precisely \eqref{eq:Dn-bound} proved by two integrations by parts, with the vanishing of $f(0)$ removing the $O(1/n)$ boundary term
from the first integration by parts.
\end{Remark}

\section{Harmonic Extensions to $\Z^{s}\times\N$}

In this section, we extend our construction to $\Z^{s}\times\N$.
 
Let $e_{j}$ denote the $j$th coordinate vector in $\Z^{s}$. Write
$\delta^{\pm}_{x_{j}}$ and $\partial_{k}^{\pm}$ for the horizontal and vertical
forward/backward differences, respectively,  so that for every $x\in\Z^{s}$ and $k\ge1$, we have:
\[
\delta^{+}_{x_{j}}F(x,k)=F(x+e_{j},k)-F(x,k),\quad
  \delta^{-}_{x_{j}}F(x,k)=F(x,k)-F(x-e_{j},k),
\]
\[
  \partial_{k}^{+}F(x,k)=F(x,k+1)-F(x,k),\quad
  \partial_{k}^{-}F(x,k)=F(x,k)-F(x,k-1).
\]

\subsection*{The $(2s{+}3)$–point Laplacian and the Poisson extension}
The fully–discrete $(2s{+}3)$-point Laplacian on $\Z^{s}\times\N$ is
\begin{equation*}\label{eq:dplus1-Lap}
  (\Lapd{F})(x,k)
  = \sum_{j=1}^{s}\big(F(x+e_{j},k)+F(x-e_{j},k)\big) + F(x,k+1)+F(x,k-1)
    -2(s{+}1)F(x,k),
\end{equation*}
so $\Lapd=\sum_{j=1}^{s}\delta^{-}_{x_{j}}\delta^{+}_{x_{j}}
+ \partial_{k}^{-}\partial_{k}^{+}$.
Given a boundary sequence $a:\Z^{s}\to\C$ at height $k=0$, we seek
$U:\Z^{s}\times\N\to\C$ with $U(\cdot,0)=a$ and $\Lapd{U}=0$.

Let $\T^{s}=(-\pi,\pi]^{s}$ and set
\[
  \widehat{a}(\theta)=\sum_{x\in\Z^{s}}a(x)\,e^{-\ii x\cdot\theta},
  \qquad
  \IFT_{s}[g](x)=\frac{1}{(2\pi)^{s}}\int_{\T^{s}}g(\theta)\,e^{\ii x\cdot\theta}\,d\theta,
  \qquad \theta=(\theta_{1},\dots,\theta_{s}).
\]
Taking the horizontal Fourier transform of $\Lapd{U}=0$ yields, for each fixed
$\theta\in\T^{s}$,
\[
  \widehat U(\theta,k+1)+\widehat U(\theta,k-1)
  +\Big(2\sum_{j=1}^{s}\cos\theta_{j}-2(s{+}1)\Big)\widehat U(\theta,k)=0.
\]
Looking for separated solutions $\widehat U(\theta,k)=C(\theta)\,\rho(\theta)^{k}$ gives
\[
  \rho+\rho^{-1}=2\Big((s{+}1)-\sum_{j=1}^{s}\cos\theta_{j}\Big).
\]
We keep the root in $(0,1]$ to ensure boundedness as $k\to\infty$:
\begin{equation}\label{eq:rho-d-explicit}
  \rho(\theta)\;=\;\alpha(\theta)-\sqrt{\alpha(\theta)^{2}-1},
  \qquad
  \alpha(\theta):=(s{+}1)-\sum_{j=1}^{s}\cos\theta_{j}\in [1,2{s}+1].
\end{equation}
Imposing $\widehat U(\theta,0)=\widehat a(\theta)$ yields
\[
  \widehat U(\theta,k)=\rho(\theta)^{k}\,\widehat a(\theta),
  \qquad
  U(x,k)=\IFT_{s}\!\big[\rho(\cdot)^{k}\widehat a\big](x)
        =\sum_{y\in\Z^{s}}P_{k}^{(s)}(x-y)\,a(y),
\]
where the $s$–dimensional discrete Poisson kernel is
\[
  P_{k}^{(s)}(x)=\frac{1}{(2\pi)^{s}}\int_{\T^{s}}\rho(\theta)^{k}\,e^{\ii x\cdot\theta}\,d\theta.
\]
Similarly as in Lemma~\ref{lem:positivity},
$P_{k}^{(s)}\ge0$, $P_{k}^{(s)}\in\ell^{1}(\Z^{s})$, and
$\sum_{x\in\Z^{s}}P_{k}^{(s)}(x)=1$ for every $k\in\N$.

\subsection*{A vector conjugate function and the discrete Cauchy–Riemann system}
For the conjugate part, we construct $s$ components $V=(V_{1},\dots,V_{s})$.

We introduce
\begin{equation*}\label{eq:omega-def}
  \omega_j(\theta)=
  \begin{cases}
  \dfrac{\sgn\theta_j\,\sin(\theta_j/2)}{\Big(\sum_{j=1}^{s}\sin^{2}\!\frac{\theta_{j}}{2}\Big)^{1/2}},& \theta\neq(0,0,...,0),\\[6pt]
  0,& \theta=0,
  \end{cases}
  \qquad j=1,\ldots,s.
\end{equation*}
The key algebraic identity now reads:
\begin{equation}\label{eq:rho-vector}
  \rho(\theta)-1
  \;=\; -\,\rho(\theta)^{1/2}\sum_{j=1}^{s}
        \omega_{j}(\theta)\,\bigl(e^{\ii\theta_{j}}-1\bigr)\,
        \bigl(-\ii\,\sgn\theta_{j}\bigr)\,e^{-\ii\theta_{j}/2},
\end{equation}
and it is derived using the identity:
\begin{equation*}\label{eq:key-d}
  \frac{(\rho(\theta)-1)^{2}}{\rho(\theta)}
  \;=\; 2\sum_{j=1}^{s}\bigl(1-\cos\theta_{j}\bigr)
  \;=\; 4\sum_{j=1}^{s}\sin^{2}\!\frac{\theta_{j}}{2},
  \qquad \theta\in\T^{s}.
\end{equation*}

We define $\widehat V_{j}$ on the Fourier side by
\begin{equation}\label{eq:Vj-hat}
  \widehat V_j(\theta,k)
  =\bigl(-\ii\,\sgn\theta_j\bigr)\,\omega_j(\theta)\,
    \rho(\theta)^{\,k+\frac12}\,e^{\ii\theta_j/2}\,\widehat a(\theta),
  \qquad \theta\in\T^{s},\ k\in\N.
\end{equation}
If we let $V_{j}(x,k)=\IFT_{s}[\widehat V_{j}(\cdot,k)](x)$, 
then each $V_{j}$ is harmonic with respect to $\Lapd$.

We obtain the following equations, analogous to the generalized Cauchy-Riemann system in higher dimensions \cite[page 120]{stein-book}. 

\begin{prop}[Discrete Cauchy–Riemann system in $\Z^{s}\times\N$]\label{prop:CR-d}
With $U$ and $V=(V_{1},\dots,V_{s})$ defined above, for every $x\in\Z^{s}$ and $k\ge1$, we have:
\begin{align}
   \partial_{k}^{+}U(x,k)   &= -\,\sum_{j=1}^{s}\delta^{-}_{x_{j}}V_{j}(x,k),\label{eq:dCR-2}\\ \delta^{+}_{x_j}U(x,k) &=\,\partial_{k}^{-}V_{j}(x,k)\qquad (j=1,\dots,s),\label{eq:dCR-3}\\
  \delta^{+}_{x_j}V_{l}(x,k) &=\,\delta^{+}_{x_l}V_{j}(x,k)\qquad (j,l=1,\dots,s). \label{eq:dCR-1}
\end{align}
\end{prop}

\begin{proof}
Using the shift rules and the definitions,
\[ \FT_{s}\!\big[\partial^{+}_{k}U\big]
   =(\rho-1)\,\rho^{k}\widehat a,\qquad
  \FT_{s}\!\big[\delta^{-}_{x_j}V_j\big]
   =(1-\e^{-\ii\theta_j})\,(-\ii\,\sgn\theta_j)\,\omega_j\,
     \rho^{k+\frac12}\,\e^{\ii\theta_j/2}\widehat a.
\]
Therefore, \eqref{eq:dCR-2} reduces to \eqref{eq:rho-vector}. The proof of \eqref{eq:dCR-3} is similar.

To prove \eqref{eq:dCR-1}, we use that for all $j$ and $l$, the following identity holds:
\begin{equation*}
\big(\e^{\ii\theta_j}-1\big)\,\bigl(-\ii\,\sgn\theta_l\bigr)
      \omega_l(\theta)\,\e^{\ii\theta_l/2}
  =\big(\e^{\ii\theta_l}-1\big)\,\bigl(-\ii\,\sgn\theta_j\bigr)
      \omega_j(\theta)\,\e^{\ii\theta_j/2}.
\end{equation*}
\end{proof}

Thus, the complex vector field $F=(U,V_{1},\dots,V_{s})$ obeys a lattice Cauchy-Riemann system whose elimination produces the $(2s{+}3)$–point Laplace equation.

\subsection*{Boundary transforms and properties}
At height $k=0$, \eqref{eq:Vj-hat} gives the boundary operators
$T_{j}$ with Fourier multipliers
\begin{align*}\label{eq:Tj-mult}
  m_j(\theta)&=\bigl(-\ii\,\sgn\theta_j\bigr)\,\omega_j(\theta)\,\rho(\theta)^{1/2}\,e^{\ii\theta_j/2}
  \;\\&=\;
  -\,\ii\,\frac{\sin(\theta_j/2)}{\Big(\sum_{j=1}^{s}\sin^{2}\!\frac{\theta_{j}}{2}\Big)^{1/2}}\,\rho(\theta)^{1/2}\,e^{\ii\theta_j/2}.
\end{align*}

Therefore, for $x=(x_1,\dots,x_s)\in\Z^{s}$, we have
\[
  (T_j a)(x)=\sum_{y\in\Z^{s}}K_{T_j}(x-y)\,a(y),
  \quad\text{where}\quad
  K_{T_j}(x)
  =\frac{1}{(2\pi)^{s}}\int_{\T^{s}}
     m_j(\theta)\,\e^{\ii x\cdot\theta}\,d\theta.
\]

We get that the kernel $K_{T_j}$ is real-valued. Moreover, for every $x\in\Z^{s}$,
\[
  K_{T_j}(-x-e_j)=-\,K_{T_j}(x),
\]
and we have a similar cancellation property as in Proposition \ref{prop:kernel-symmetry}: $\sum_{x\in\Z^{s}}K_{T_j}(x)=0$.

Consider the following Riesz transforms $R_{j}^{+}$ on $\Z^s$, defined by their Fourier multipliers
\begin{equation*}
  \widehat{R_{j}^{+}a}(\theta)
  = m_{R_{j}}(\theta)\,\widehat a(\theta),\qquad
  m_{R_{j}}(\theta):=\bigl(-\ii\,\sgn\theta_{j}\bigr)\,\omega_{j}(\theta)\,e^{\ii\theta_{j}/2}.
\end{equation*}
These operators are Calderón–Zygmund operators on the discrete
space $\Z^{s}$. Indeed, the size and smoothness estimates for their corresponding kernels follow from the expansions of the multipliers near $\theta=0$:
$m_{R_j}(\theta)
   = -\ii\,\frac{\theta_j}{|\theta|}+O(|\theta|).$

As before, we obtain mapping properties of $T_{j}$ by comparing them to the transforms $R_j^+$. For $1<p<\infty$ and weights $w\in A_{p}(\Z^{d})$, the operators are bounded on $\ell^p(w; \Z^s)$.
Moreover, when $w\in A_{1}(\Z^{d})$, each $T_{j}$ is of weak type
$(1,1)$; that is, $T_j:\ell^1(w; \Z^s) \to \ell^{1,\infty}(w; \Z^s)$.

\section*{Acknowledgment}

The author would like to thank Mar\'ia Cristina Pereyra and Brett D. Wick for useful suggestions and discussions.

\bibliographystyle{amsplain}

\end{document}